\documentclass{amsart}

\usepackage{amssymb}
\usepackage{url, color}

\newtheorem{theorem}{Theorem}[section]
\newtheorem{proposition}[theorem]{Proposition}
\newtheorem{lemma}[theorem]{Lemma}
\newtheorem{corollary}[theorem]{Corollary}

\theoremstyle{definition}

\theoremstyle{remark}

\numberwithin{equation}{section}

\begin{document}

\title[holomorphic extension theorem on hyperconvex domains]{An Extension theorem of holomorphic functions on hyperconvex domains}

\author{Seungjae Lee}
\address{Department of Mathematics, Pohang University of Science and Technology, Pohang, 790-784,
Republic of Korea}
\email{seungjae@postech.ac.kr}
\thanks{}

\author{Yoshikazu Nagata}
\address{Graduate School of Mathematics, Nagoya University, Furocho,
Chikusaku, Nagoya, 464-8602, Japan}
\email{m10035y@math.nagoya-u.ac.jp}


\date{}

\dedicatory{}


\commby{}

\begin{abstract}
Let $n \geq 3$ and $\Omega$ be a bounded domain in $\mathbb{C}^n$ with a smooth negative plurisubharmonic exhaustion function $\varphi$. As a generalization of Y. Tiba's result, we prove that any holomorphic function on a connected open neighborhood of the support of $(i\partial \bar \partial \varphi )^{n-2}$ in $\Omega$  can be extended to the whole domain $\Omega$. To prove it, we combine an $L^2$ version of Serre duality and Donnelly-Fefferman type estimates on $(n,n-1)$- and $(n,n)$- forms.
\end{abstract}

\maketitle

\section{Introduction}

In this article, we study a kind of the Hartogs extension theorem, which appears in Y. Tiba's paper \cite{Tiba}. The Hartogs extension theorem states that any holomorphic function on $\Omega \setminus K$, where $\Omega$ is a domain in $\mathbb{C}^n$, $n>1$, $K$ is a compact set in $\Omega$
and $\Omega\setminus K$ is connected, extends holomorphically on the whole domain $\Omega$.

This phenomenon is different from the case of the function theory of one complex variable, and have become a starting point of the function theory of several complex variables. For the several complex variables, the notion of the (strict) pseudoconvexity for the boundary
of a given domain have become crucial. Let $\Omega$ be a smoothly bounded pseudoconvex domain. Denote by $A(\Omega)$ the uniform algebra of functions that are holomorphic on $\Omega$ and continuous on $\overline{\Omega}$. The Shilov boundary of $A(\Omega)$ is the smallest closed subset $S(\Omega)$ in $\partial \Omega$ on which the maximum value of $|f|$ coincides with that on $\overline{\Omega}$ for every function $f$ in $A(\Omega)$. In fact, the Shilov boundary of $A(\Omega)$ is the closure of the set of strictly pseudoconvex boundary points of $\Omega$ (see \cite{Basener}).
By \cite{Sibony}, it is known that any holomorphic function $f$ on $\overline{\Omega}$ can be represented as $f(x)=\int f(z) d\mu_x(z)$ where $d\mu_x$ is a measure supported on the Shilov boundary $S(\Omega)$.

Assume further that $\Omega$ has a negative smooth plurisubharmonic function $\varphi$ on $\Omega$ such that $\varphi \rightarrow 0$ when $z \rightarrow \partial \Omega$. Denote by $\textup{Supp}(i \partial \bar \partial \varphi)^{k}$ the support of $(i \partial \bar \partial \varphi)^{k}$. By  \cite{Bedford}, it can be shown that, for small $\epsilon >0$, the Shilov boundary $S(\Omega_{\epsilon})$ of $\Omega_{\epsilon} = \{ \varphi <-\epsilon \}$ is a subset of $\textup{Supp}(i\partial \bar \partial \varphi)^{n-1}$.

In this context, it is natural to ask whether any holomorphic function on the support of $(i \partial \bar \partial \varphi)^{n-1}$ can be extended to the whole domain $\Omega$. With this motivation, Y. Tiba proved the following theorem in \cite{Tiba}.

\begin{theorem}
Let $\Omega$ be a bounded domain in $\mathbb{C}^n$, $n \geq 4$. Suppose that $\varphi \in C^{\infty}(\Omega)$ is a negative plurisubharmonic function which satisfies $\varphi(z) \rightarrow 0$ as $z \rightarrow \partial \Omega$. Let $V$ be an open connected neighborhood of
$\textup{Supp}(i \partial \bar \partial \varphi)^{n-3}$ in $\Omega$. Then any holomorphic function on $V$ can be extended to $\Omega$.
\end{theorem}

In \cite{Tiba}, Y. Tiba proved Donnelly-Fefferman type estimates for $(0,1)$- and $(0,2)$-forms and used them for establishing suitable
$L^2$ estimates of $\bar\partial$-equations. In this process, the Donnelly-Fefferman type estimate of $(0,2)$-form and an integrability condition contribute to appear the restriction of the power $n-3$ in Theorem 1.1.

In this article, we use the Donnelly-Fefferman type estimates for $(n,n)$- and $(n,n-1)$- forms rather than $(0,1)$- and $(0,2)$-forms. In this case, the restriction of $n-3$ is changed by $n-2$, and it improves Theorem 1.1. Finally, using dualities between $L^2$-Dolbeault cohomologies, in the same way as \cite{Ohsawa, Ohsawa1}, we can simplify Y. Tiba's proof and obtain the generalized result:

\begin{theorem}
Let $\Omega$ be a bounded domain in $\mathbb{C}^n$, $n \geq 3$. Suppose that $\varphi \in C^{\infty}(\Omega)$ is a negative plurisubharmonic function which satisfies $\varphi(z) \rightarrow 0$ as $z \rightarrow \partial \Omega$. Let $V$ be an open connected neighborhood of
$\textup{Supp}(i \partial \bar \partial \varphi)^{n-2}$ in $\Omega$.
Then any holomorphic function on $V$ can be extended to $\Omega$.
\end{theorem}

If the boundary is smooth, then Lemma 4.3 and the proof of Theorem 1.2 show the following.

\begin{corollary}
Let $\Omega$ be a smoothly bounded domain in $\mathbb{C}^n, n \geq 3$, with a smooth plurisubharmonic defining function $\varphi$ on a neighborhood of $\overline\Omega$. Let $S$ be the closure of the subset $\{ z \in \partial \Omega : ~\text{\rm the Levi form of} ~\varphi~ \text{\rm at}~$z$  \text{ \rm is of rank at least}~ n-2 \}$ in $\partial \Omega$. Then for any connected open neighborhood $V$ of $S$ in $\overline\Omega$, any holomorphic function on $V \cap \Omega$ which is continuous on $\overline{V \cap \Omega}$ can be extended to $\Omega$.
\end{corollary}

By using a convergence sequence of smooth plurisubharmonic functions to $\varphi$, we also obtain the following corollary which is the improvement of Corollary 1 of \cite{Tiba}.

\begin{corollary}
Let $\Omega$ be a bounded domain in $\mathbb{C}^n$, $n \geq 3$. Suppose that $\varphi \in C^{0}(\Omega)$ is a negative plurisubharmonic function which satisfies $\varphi(z) \rightarrow 0$ as $z \rightarrow \partial \Omega$. Let $V$ be an open connected neighborhood of $\textup{Supp}(i \partial \bar \partial \varphi)$ in $\Omega$. Then any holomorphic function on $V$ can be extended to $\Omega$.
\end{corollary}

\section{Preliminaries}

In this section, we review $L^2$ estimates of $\bar\partial$-operators and introduce some notations which are used in this paper. Let $\Omega \subset \mathbb{C}^n$ be a domain, and let $\omega$ be a K\"{a}hler metric on $\Omega$. We denote by $| \cdot |_{\omega}$ the norm of $(p,q)$-forms induced by $\omega$ and by $dV_{\omega}$ the associated volume form of $\omega$. Then, we denote by $L^2_{p,q}(\Omega, e^{-\varphi}, \omega)$ the Hilbert space of measurable $(p,q)$-forms $u$ which satisfy
$$
||u ||^2_{\omega, \varphi } = \int_{\Omega} |u|^2_{\omega} e^{-\varphi} dV_{\omega} < \infty.
$$

Let $\bar \partial : L^2_{p,q}(\Omega, e^{-\varphi},\omega) \rightarrow L^2_{p,q+1} (\Omega, e^{-\varphi}, \omega)$ be the closed densely defined linear operator, and $\bar \partial^{*}_{\varphi}$ be the Hilbert space adjoint of the $\bar \partial$-operator.
We denote the $L^2$-Dolbeault cohomology group as $H^2_{p,q}(\Omega,e^{-\varphi}, \omega)$ and the space of Harmonic forms as
\begin{equation*}
\mathcal{H}_{p,q}^{2} (\Omega,e^{-\varphi},\omega) = L^2_{p,q} (\Omega, e^{-\varphi}, \omega) \cap \text{Ker} \bar \partial \cap \text{Ker} \bar \partial^{*}_{\varphi}.
\end{equation*}
It is known that, if the image of the $\bar\partial$-operator is closed, then
these two spaces are isomorphic:
\begin{equation*}
H^2_{p,q}(\Omega,e^{-\varphi}, \omega)\cong
\mathcal{H}_{p,q}^{2} (\Omega,e^{-\varphi},\omega).
\end{equation*}


Let $\lambda_1 \leq \cdots \leq \lambda_n$ be the eigenvalues of $i \partial \bar \partial \varphi$ with respect to $\omega$ then we have
\begin{equation*}\label{eq2}
\langle [i \partial \bar \partial \varphi, \Lambda_{\omega}]u,u \rangle_{\omega} ~ \geq ~  (\lambda_1 + \cdots + \lambda_{q} - \lambda_{p+1} \cdots -\lambda_{n}) \langle u,u \rangle_{\omega}
\end{equation*}
for any smooth $(p,q)$-form. Here, $\Lambda_\omega$ is the adjoint of left multiplication by $\omega$.

Suppose that $A_{\omega,\varphi}=[i \partial \bar \partial \varphi, \Lambda_{\omega}]$ is positive definite and $\omega$ is a K\"{a}hler metric.
By \cite{Demailly}, if $\Omega$ is a pseudoconvex domain, then for any $\bar \partial$-closed form $f \in L^2_{n,q}(\Omega,e^{-\varphi}, \omega)$, there exists a $u \in L^2_{n,q-1}(\Omega,e^{-\varphi}, \omega)$ such that $\bar \partial u = f$ and
\begin{equation}\label{eq3}
\int_{\Omega} |u|_{\omega}^2 e^{-\varphi} dV_{\omega} \leq \int_{\Omega} \langle A^{-1}_{\omega,\varphi} f, f \rangle_{\omega} e^{-\varphi} dV_{\omega}.
\end{equation}

\section{Donnelly-Fefferman type estimates for $(n,q)$-forms}

Let $\Omega \subset \mathbb{C}^n$ be a bounded domain with a negative plurisubharmonic function $\varphi \in C^{\infty}(\Omega)$ such that $\varphi \rightarrow 0$ as $z \rightarrow \partial \Omega$. Consider a smooth
strictly plurisubharmonic function $\psi$ on $\overline{\Omega}$. Since $\phi = -\log(-\varphi)$ is a plurisubharmonic exhaustion function on
$\Omega$ and $|\bar\partial \phi|^2_{i \partial \bar \partial \phi} \leq 1$, $\omega = i\partial \bar \partial(\frac{1}{2n} \psi + \phi)$ is a complete K\"{a}hler metric on $\Omega$.
Let $A_{\omega,\delta}$ be $[i\partial \bar \partial (\psi + \delta \phi) , \Lambda_{\omega}]$ if $\delta \geq 0$.

\begin{lemma}\label{Lem1}
Suppose that $0<\delta<q, ~1 \leq q \leq n$. Then for any $\bar \partial$-closed form $f \in L^2_{n,q}(\Omega, e^{-\psi +\delta \phi}, \omega)$,
there exists a solution $u \in L^2_{n,q-1} (\Omega, e^{-\psi +\delta\phi}, \omega)$ such that $\bar \partial u = f $  and
\begin{equation}\label{L1E1}
\int_{\Omega} |u|^2_{\omega} e^{-\psi + \delta \phi} dV_{\omega} \leq C_{q, \delta} \int_{\Omega} \langle A^{-1}_{\omega, \delta} f, f \rangle_{\omega} e^{-\psi + \delta \phi} dV_{\omega}
\end{equation}
where $C_{q,\delta}$ is a constant which depends on $q,\delta$.
\end{lemma}

By Lemma 3.1, the $L^2$-Dolbeault cohomology group $H^{2}_{n,q}(\Omega, e^{-\psi + \delta\phi},\omega)$ vanishes.

\begin{corollary}\label{cor1}
Under the same condition as Lemma \ref{Lem1}, $H^{2}_{n,q}(\Omega, e^{-\psi + \delta\phi},\omega)=\{0\}$.
\end{corollary}

To prove Lemma 3.1, we use the idea of Berndtsson--Charpentier's proof of the Donnelly-Fefferman type estimate in \cite{Berndtsson}.

\begin{proof}
Since $\Omega$ can be exhausted by pseudoconvex domains $\Omega_k \subset \subset \Omega$, for any $\bar\partial$--closed form $f\in L^2_{n,q} (\Omega, e^{-\psi}, \omega)$, the minimal solution $u_k \in L^2_{n,q-1} (\Omega_k, e^{-\psi} , \omega)$ of $\bar \partial u_k = f$ exists and it satisfies
$$
\int_{\Omega_k} |u_{k}|^2_{\omega} e^{-\psi} dV_{\omega} \leq \int_{\Omega_k} \langle A^{-1}_{\omega,0} f, f \rangle_{\omega}e^{-\psi}dV_{\omega}.
$$
Note that $A_{\omega,0}$ has its inverse $A^{-1}_{\omega,0}$ on $\Omega_k$ by the plurisubharmonicity of $\phi$.

We consider $u_{k} e^{\delta \phi}$. Since $\phi$ is bounded on $\Omega_k$, $u_k e^{\delta \phi} \in  L^2_{n,q-1} (\Omega_k, e^{-\psi - \delta \phi},\omega)$ and it is orthogonal to $N_{n,q-1}$ where $N_{n,q-1}$ is the kernel of
$$
\bar\partial : L^2_{n,q-1} (\Omega_k, e^{-\psi - \delta\phi}, \omega) \rightarrow L^2_{n,q} (\Omega_k, e^{-\psi - \delta \phi}, \omega).
$$
By $ |\partial \phi|^2_{\omega} \leq 1$, we have $(f  + \bar \partial \phi \wedge \delta u_k)e^{\delta \phi} \in L^2_{n,q}(\Omega_k, e^{-\psi - \delta \phi}, \omega)$.
Therefore, $u_ke^{\delta \phi}$ is the minimal solution of
\begin{equation*}
\bar\partial (u_ke^{\delta\phi})=(f  + \bar \partial \phi \wedge \delta u_k )e^{\delta \phi} \in L^2_{n,q}(\Omega_k, e^{-\psi - \delta\phi}, \omega).
\end{equation*}
Thus, we have
\begin{equation*}
\int_{\Omega_k} |u_k|^2_{\omega} e^{-\psi + \delta\phi} dV_{\omega} \leq \int_{\Omega_k} \langle A^{-1}_{\omega, \delta}(f + \bar \partial \phi \wedge \delta u_k ), f + \bar \partial \phi \wedge \delta u_k \rangle_{\omega} e^{-\psi +\delta\phi}dV_{\omega}.
\end{equation*}
By the Cauchy-Schwarz inequality, for any $t>0$,
\begin{align*}
& \int_{\Omega_k} \langle A^{-1}_{\omega, \delta }(f + \bar \partial \phi \wedge \delta u_k ), f + \bar \partial \phi \wedge \delta u_k \rangle_{\omega} e^{-\psi +\delta\phi}dV_{\omega} \\
& \leq  \bigg( 1+ \frac{1}{t} \bigg) \int_{\Omega_k} \langle A^{-1}_{\omega,\delta} f, f \rangle _{\omega} e^{-\psi + \delta \phi} dV_{\omega} \\
& + (1+t) \delta^2 \int_{\Omega_k} \langle A^{-1}_{\omega,\delta} (\bar \partial \phi \wedge u_k), \bar \partial \phi \wedge u_k \rangle_{\omega} e^{-\psi + \delta \phi}dV_{\omega}.
\end{align*}
Since $\omega = i \partial \bar \partial (\frac{1}{2n} \psi + \phi)$
and $\delta <2n$,
we have $i\partial \bar \partial (\psi + \delta \phi) \geq \delta i \partial \bar \partial (\frac{1}{2n}\psi + \phi)$. Hence, $\langle A_{\omega,\delta} f,f \rangle_{\omega} \geq q \delta |f|^2_{\omega}$ if $f$ is an $(n,q)$-form.
If we take $t$ sufficiently close to 0,
then $C_{q,\delta}:=(1+\frac{1}{t})/(1-(1+t)\frac{\delta}{q})$ is positive since $\delta < q$.
Note that $C_{q,\delta}$ does not depend on $k$.
For such a $t$, we have
\begin{eqnarray}\label{L1E2}
 \int_{\Omega_k} |u_k|^2_{\omega} e^{-\psi + \delta\phi} dV_{\omega} \leq C_{q,\delta}\int_{\Omega} \langle A^{-1}_{\omega,\delta} f ,f \rangle_{\omega} e^{-\psi + \delta\phi} dV_{\omega}.
\end{eqnarray}

Note that $L^2_{n,q}(\Omega,e^{-\psi + \delta\phi},\omega)$ is a subset of $L^2_{n,q}(\Omega,e^{-\psi},\omega)$. Hence, for any $\bar \partial$-closed form $f \in L^2_{n,q} (\Omega, e^{-\psi+\delta\phi}, \omega)$,
the right-hand side of $(\ref{L1E2})$ is finite.
Since $\{\Omega_k\}$ is an exhaustion of $\Omega$ and $u_{k}$ is uniformly bounded by (\ref{L1E2})
, we obtain a weak limit $u \in L^{2,\text{loc}}_{n,q-1}(\Omega, e^{-\psi +\delta\phi}, \omega)$ of $u_{k}$, and it satisfies $\bar \partial u = f$ and, for each compact set $K$ in $\Omega$,
\begin{equation*}
 \int_{K} |u|^2_{\omega} e^{-\psi + \delta\phi} dV_{\omega} \leq C_{q,\delta}\int_{\Omega} \langle A^{-1}_{\omega,\delta} f ,f \rangle_{\omega} e^{-\psi + \delta\phi} dV_{\omega}.
\end{equation*}
Using the monotone convergence theorem for $K$, we obtain the desired result.
\end{proof}

\section{Proof of the main Theorem 1.2}
The key proposition of this section is the following:

\begin{proposition}
Under the same condition as Theorem 1.2, if $1 \leq q < n$ and $0<\delta<n-q$
then $H^{2}_{0,q} (\Omega, e^{\psi-\delta\phi}, \omega)=\{0\}$.
\end{proposition}

\begin{proof}
Corollary \ref{cor1} implies that $H^2_{n,n-q}(\Omega, e^{-\psi+\delta\phi},\omega)$ and $H^2_{n,n-q+1}(\Omega, e^{-\psi +\delta\phi} ,\omega)$ are $\{0\}$.
Therefore, the Serre duality in \cite{Shaw} implies that
\begin{equation*}
\bar \partial:L^2_{0,q-1} (\Omega, e^{\psi - \delta \phi}, \omega) \rightarrow L^2_{0,q} (\Omega, e^{\psi - \delta \phi}, \omega)
\end{equation*}
has a closed range and
\begin{equation*}
\mathcal{H}_{n,n-q}^2 (\Omega, e^{-\psi + \delta \phi}, \omega) \cong \mathcal{H}_{0,q}^2 (\Omega, e^{\psi - \delta \phi}, \omega)=\{0\}
\end{equation*}
since $\omega$ is a complete K\"{a}hler metric on $\Omega$. Hence, $H^2_{0,q}(\Omega, e^{\psi - \delta \phi},\omega)=\{0\}$.
\end{proof}

For the convenience of readers, we repeat Lemma 5 in \cite{Tiba}.
\begin{lemma}\label{T1}
Let $\Omega$ be a smoothly bounded pseudoconvex domain in $\mathbb{C}^n$
with a plurisubharmonic defining function $\varphi\in C^\infty(\overline\Omega)$, i.e. $\Omega=\{\varphi<0\}$ and $d\varphi(z) \neq 0$ on $\partial\Omega$. Let $p\in \partial\Omega$ and let
$1 \leq k \leq n$ be an integer. Assume that $(i\partial\bar\partial\varphi)^k = 0$ in a neighborhood of $p$. If $\delta > k$,
then  $\exp(-\delta\phi)dV_{\omega}$ is integrable around $p$ in $\Omega$.
Here, $\phi = -\log(-\varphi)$ and $\omega = i\partial \bar \partial(\frac{1}{2n} \psi + \phi)$.
\end{lemma}

The following lemma is a variant of Lemma 5 in \cite{Tiba}. It is used to prove Corollary 1.3.
\begin{lemma}\label{}
Let $\Omega$ be a smoothly bounded pseudoconvex domain in $\mathbb{C}^n$
with a smooth plurisubharmonic defining function $\varphi$. Let $p\in \partial\Omega$ and let
$1 \leq k \leq n-1$ be an integer.
Assume that
\begin{equation}\label{T1E1}
(i \partial \bar \partial \varphi)^{k} \wedge \partial \varphi \wedge \bar \partial \varphi =0
\end{equation}
in a neighborhood of $p$ in $\partial\Omega$, i.e. the Levi form of $\varphi$ is of rank less than k.
If $\delta > k$, then $\exp(-\delta\phi)dV_{\omega}$ is integrable around $p$ in $\Omega$.
Here, $\phi = -\log(-\varphi)$ and $\omega = i\partial \bar \partial(\frac{1}{2n} \psi + \phi)$.
\end{lemma}

\begin{proof}
Denote by $\vec{\nu}$ the unit outward normal vector at $p$.
For a point $q\in \partial \Omega$ near $p$, (\ref{T1E1}) implies that the Levi form of $\varphi$ at $q$ has at least $n-k$ zero eigenvalues. Now consider a holomorphic coordinate system $(z_1, \cdots, z_n)$ such that $i\partial \bar \partial \varphi = \sum_{ij} a_{ij} dz_i  \wedge d \bar{z_j}$ and $i\partial \bar \partial \varphi|_{q} = \sum_{i} a_{ii}(q) dz_i  \wedge d \bar{z_i}$, where $a_{ij}$ is a smooth function on $\overline{\Omega}$ and $a_{ii}(q)=0$ when $1 \leq i \leq n-k$.

By smoothness of $a_{ij}$, it follows that $a_{ij}(q-t \vec{\nu}) = O(t)$ if $i \neq j$ or ~$1 \leq i=j \leq n-k$. Hence,
\begin{equation*}
(i\partial \bar \partial \varphi)^{k+l} \wedge \partial \varphi \wedge \overline{\partial} \varphi \wedge (i \partial \bar \partial \psi)^{n-k-l-1} = O(t^{l+1})(i \partial \bar \partial \psi)^{n}
\end{equation*}
and
\begin{equation*}
(i \partial \bar \partial \varphi)^{k+l} \wedge (i \partial \bar \partial \psi)^{n-k-l} = O(t^{l})(i \partial \bar \partial \psi)^{n}
\end{equation*}
on the real half line $q-t \vec{\nu}$, $t>0$.
Since $\exp(-\delta \phi) dV_{\omega}$ is approximately
\begin{equation*}
(-\varphi)^{\delta} \left(\frac{(i\partial \bar \partial \varphi)^{n-1} \wedge \partial \varphi \wedge \overline{\partial} \varphi}{(-\varphi)^{n+1}}  +
\frac{(i \partial \bar \partial \varphi)^n + (i\partial \bar \partial \varphi)^{n-2} \wedge \partial \varphi \wedge \bar \partial \varphi \wedge (i\partial \bar \partial \psi)^{1}}{(-\varphi)^{n}}  \right),
\end{equation*}
$\exp(-\delta \phi) dV_{\omega}$ is integrable near $p$ by the Fubini theorem.
\end{proof}

We also need the following lemma. It is similar to the Lemma 5.1 of \cite{Ohsawa}.
\begin{lemma}\label{L2}
Let $\Omega$ be a smoothly bounded domain in $\mathbb{C}^n$ with a defining function $\varphi\in C^{\infty}(\overline{\Omega})$.
Let $U$ be an open set in $\mathbb{C}^n$ such that $ \partial \Omega \cap U  \not = \emptyset$ and $\Omega \cap U$ is connected. If $u$ is a holomorphic function on $\Omega \cap U$ such that
\begin{equation}\label{L2E1}
\int_{\Omega \cap U} |u|^2 {(-\varphi)^{\alpha}} dV <\infty
\end{equation}
for some $\alpha \leq -1$, then $u=0$ on $\Omega \cap U$.
\end{lemma}
\begin{proof}
Take a point $p \in \partial \Omega \cap U$. Consider a holomorphic coordinate such that $p=0$ and the unit outward normal vector $\vec{\nu}$ to $\partial \Omega$ at $p$ is $(0,\cdots,0,1)$. Take an open ball ${B}(p,r)$ which is centered at $p$ with sufficiently small radius $r>0$ such that ${B}(p,r) \subset U$.

Denote by $Z_q $ the complex line $ \{q + \lambda \vec{\nu} : \lambda \in \mathbb{C}\}$ for each $q \in \partial \Omega \cap {B}(p,r)$. By the Fubini theorem and $(\ref{L2E1})$,
\begin{equation*}
\int_{q \in E} \bigg( \int_{Z_q \cap \Omega \cap {B}(p,r)} |u|^2 (-\varphi)^{\alpha} d\lambda \bigg ) d\sigma \leq \int_{\Omega \cap {B}(p,r)} |u|^2 (-\varphi)^{\alpha} dV,
\end{equation*}
where $E\subset \partial\Omega\cap {B}(p,r)$ is a local parameter space for $Z_q$ of finite measure.
Note that $E$ can be chosen as $(2n-2)$-dimensional smooth surface in $\partial \Omega$ and any fiber $(p+\mathbb{C} \vec{\nu}) \cap \partial \Omega$ for $p\in E$ is transversal to $E$.
Then, we have
\begin{equation}\label{Integrable}
\int_{Z_q \cap \Omega \cap {B}(p,r)} |u|^2 (-\varphi)^{\alpha} d\lambda < \infty
\end{equation}
for $q \in E$ almost everywhere.
Since $Z_{q'}=Z_q$ if $q' \in Z_{q} \cap \partial \Omega$,
there exists an connected open set $V$ in $\partial \Omega$ such that
(\ref{Integrable}) holds
for $q \in \partial\Omega\cap V$ almost everywhere.
For such $q$, since $\alpha \leq -1$ and $u$ is holomorphic on $\Omega \cap U$, by Lemma 5.1 of \cite{Ohsawa}, $u=0$ on $Z_q \cap \Omega \cap {B}(p,r)$. Therefore, $u=0$ on $\Omega \cap U$.
\end{proof}

\emph{Proof of the Theorem 1.2.}
First, we assume that $\partial \Omega$ is smooth, $\varphi$ is smooth plurisubharmonic on $\overline{\Omega}$, $d\varphi\neq 0$ on $\partial\Omega$ and the distance between $\partial V\cap\Omega$ and $\textup{Supp}(i \partial \bar \partial \varphi)^{n-2}$ is positive. Take a neighborhood W of $\textup{Supp}(i \partial \bar \partial \varphi)^{n-2}$ such that
$W$ is contained in $V$, the distance between $\partial W\cap\Omega$ and $\textup{Supp}(i \partial \bar \partial \varphi)^{n-2}$ is positive, and $\partial V\cap \partial W\cap\Omega=\emptyset$.
Consider a real-valued smooth function $\chi$ on $\Omega$ which is equal to one on $\textup{Supp}(i \partial \bar \partial \varphi)^{n-2}$ and equal to zero on $\Omega-W$.

Choose $n-2<\delta<n-1$. By Lemma \ref{T1}, $\bar\partial(\chi f) \in L^2_{0,1}(\Omega, e^{\psi - \delta\phi}, \omega)$. Applying Proposition 4.1, we can find a function $u\in L^2(\Omega, e^{\psi - \delta\phi}, \omega)$ such that $\bar \partial u = \bar \partial (\chi f)$

Take a strictly pseudoconvex point $p\in \partial \Omega$ and a connected open set $U$ such that $p \in U \cap \partial \Omega \subset \Omega \cap \textup{Supp}\big( (i \partial \bar \partial \varphi)^{n-1} \wedge \partial \varphi \wedge \bar \partial \varphi \big)$.
Since $\omega = i (\partial \bar \partial \psi + \frac{\partial \bar \partial \varphi}{-\varphi} + \frac{\partial \varphi \wedge \bar \partial \varphi}{\varphi^2})$, we have
\begin{equation}\label{M2}
\int_{\Omega \cap U} |u|^2 e^{\psi-\delta \phi} (-\varphi)^{-(n+1)} dV \lesssim \int_{\Omega} |u|^2 e^{\psi - \delta \phi} dV_{\omega} <\infty.
\end{equation}
Since $\bar \partial u = 0$ on $\Omega \cap U$, by Lemma \ref{L2}, $u=0$ on $\Omega \cap U$. Therefore, the holomorphic function $\chi f - u$ on $\Omega$ coincides with $f$ on $V$ by the uniqueness of analytic continuation.

To prove the general case, we consider the subdomain $\Omega_{\epsilon} = \{ \varphi < -\epsilon \}$ of $\Omega$ with smooth boundary, where $\epsilon>0$. If $\epsilon$ is sufficiently small, $f$ has the holomorphic extension for each $\Omega_\epsilon$ by the previous argument. Due to analytic continuation, we have the desired holomorphic extension of $f$ on $\Omega$.
\begin{flushright}
$\Box$
\end{flushright}

\section*{Acknowledgments}
The first named author wishes to express gratitude to professor Kang-Tae Kim for his guidance. He is also grateful to professor Masanori Adachi for a discussion on this program. We are thankful to professor Yusaku Tiba and professor Masanori Adachi for comments to a manuscript of this article. We would also like to thank referees for their kind comments which were helpful for the improvement of the presentation of our paper. The work is a part of the first named author's Ph.D. thesis at Pohang University of Science and Technology.

\bibliographystyle{amsplain}

\end{document}